\documentclass[twoside,leqno,10pt]{amsart}
\usepackage{amsfonts}
\usepackage{amsmath}
\usepackage{amscd}
\usepackage{amssymb}
\usepackage{amsthm}
\usepackage{amsrefs}
\usepackage{latexsym}
\usepackage{color}
\usepackage{hyperref}
\hypersetup{
  colorlinks = true,
  linkcolor = blue,
  anchorcolor = blue,
  citecolor = red,
    }
\newtheorem{theorem}[subsection]{Theorem}
\newtheorem{lemma}[subsection]{Lemma}
\newtheorem{definition}[subsection]{Definition}
\numberwithin{equation}{section}

\newcommand{\intz}{\mathbb{Z}}

\newcommand{\natn}{\mathbb{N}}

\newcommand{\rear}{\mathbb{R}}

\newcommand{\norm}[1]{\left\lvert#1\right\rvert}
\title{Solving $p$-adic polynomial equations using Jarratt's Method}
\author{
Stephan Baier,
Swarup Kumar Das \and 
Saayan Mukherjee 
}
\address{Stephan Baier,
Ramakrishna Mission Vivekananda Educational and Research Institute, Department of Mathematics, G. T. Road, PO Belur Math, Howrah, West Bengal 711202, India}
\email{stephanbaier2017@gmail.com}
\address{Swarup Kumar Das,
Ramakrishna Mission Vivekananda Educational and Research Institute, Department of Mathematics, G. T. Road, PO Belur Math, Howrah, West Bengal 711202, India}
\email{313swarup@gmail.com}
\address{Saayan Mukherjee,
Ramakrishna Mission Vivekananda Educational and Research Institute, Department of Mathematics, G. T. Road, PO Belur Math, Howrah, West Bengal 711202, India}
\email{saayanwith@gmail.com}
\begin{document}
\maketitle
\begin{abstract}
We implement an iterative numerical method to solve polynomial equations $f(x)=0$ in the $p$-adic numbers, where $f(x) \in\mathbb{Z}_p[x]$. This method is a simplified $p$-adic analogue of Jarratt's method for finding roots of functions over the real numbers. We establish that our method has a higher order of convergence than J.F.T. Rabago's \cite{j.rabago} $p$-adic version of Olver's method from 2016. Moreover, we weaken the initial conditions in Rabago's method, which allows us to start the iteration with a multiple root of the congruence $f(x) \equiv 0 \bmod{p}$.
\end{abstract}
\tableofcontents
\section{Introduction}
Solving polynomial equations is one of the oldest problems in mathematics which dates back to the Babylonians. There are several analytic, algebraic and geometric approaches for finding roots of polynomials. In this paper we discuss an iterative method to approximate solutions of polynomial equations of the form $f(x)=0$ in the $p$-adic numbers, where $f\in \mathbb{Z}_p[x]$. 

Well-known numerical methods to approximate roots of functions on the real numbers are Newton's Method, Olver's Method \cite{olver-method}, Halley's method \cite{halley-method} and others. In this paper, we look at a method developed by P. Jarratt in 1966 \cite{jarratt}. This method has a higher order of convergence than all of the afore-mentioned methods. We set up a $p$-adic analogue of Jarratt's method. Analogues of Olver's and Halley's method in the $p$-adic setting have previously been established by J.F.T. Rabago \cite{j.rabago}. Apart from improving the order of convergence, we also weaken the initial conditions in Rabago's work. Below is a definition of the term ``order of convergence'' for general Banach spaces, including the complete valued fields $\mathbb{R}$ and $\mathbb{Q}_p$.

\begin{definition}
Let $V$ be a Banach space with norm $|.|$. Assume that $(x_n)$ is a sequence in $V$ converging to $\gamma\in V$. Then we say that $(x_n)$ converges with order $q$ if
\begin{equation*}
 R_n=\frac{\norm{x_{n+1}-\gamma}}{\norm{x_n - \gamma}^q} 
\end{equation*}
is a bounded sequence in $\mathbb{R}$.
\end{definition} 

Before turning to Jarret's method, let us briefly review Newton's and Olver's methods in the context of $p$-adic numbers. Suppose $p$ is a prime and let $f(x) \in \mathbb{Z}_p[x]$ be a polynomial. Throughout the sequel, we assume the leading coefficient of $f$ to be a unit in $\mathbb{Z}_p$. Then all solutions of $f$ in $\mathbb{Q}_p$ lie in $\mathbb{Z}_p$. Assume that $x_1 \in \mathbb{Z}_p$ such that $f(x_1) \equiv 0 \bmod{p}$ but $f'(x_1) \not\equiv 0 \bmod{p}$. Newton's method uses the iteration  
\begin{equation}\label{newton}
    x_{n+1} = x_n - \frac{f(x_n)}{f'(x_n)} \quad \mbox{for } n\ge 1,
\end{equation}
which yields a sequence in $\mathbb{Z}_p$ which converges to a simple root of $f$ in $\mathbb{Z}_p$. In the case when $f\in \mathbb{Z}[x]$, this provides a convenient proof of Hensel's lemma which, for all $m\in \mathbb{N}$, asserts that a solution $x_1\in \mathbb{Z}$ of the congruence $f(x_1)\equiv 0 \bmod{p}$ satisfying $f'(x_1)\not\equiv 0 \bmod{p}$ can be uniquely lifted to a root $x$ of the congruence $f(x)\equiv 0 \bmod{p^m}$. In fact, one may replace the condition $f'(x_1)\not\equiv 0 \bmod{p}$ by the weaker condition 
\begin{equation} \label{condi}
|f(x_1)|_p<|f'(x_1)|_p^2
\end{equation} 
which leads to a generalized version of Hensel's lemma (see \cite{p-adic_book}). The order of convergence of Newton's method is 2.

Olver's method, an iteration with higher order of convergence, was carried over to $p$-adic numbers by J.F.T. Rabago in 2016 \cite{j.rabago}. He proved that if $f(x) \in \mathbb{Z}_p[x]$ and if $x_1 \in \mathbb{Z}$ satisfies $f(x_1) \equiv 0 \bmod{p}$ and $f'(x_1) \not\equiv 0 \bmod{p}$, then the sequence $(x_n)$ defined iteratively as
\begin{equation}
    x_{n+1} = x_n - \frac{f(x_n)}{f'(x_n)} - \frac{1}{2}\cdot\frac{f(x_n)^2f''(x_n)}{f'(x_n)^3}
\end{equation}
converges to a simple root of $f$ in $\mathbb{Z}_p$. The order of convergence of this method is 3 for $p\geq3$ and 4 for $p=2$. 

Our goal is to implement a $p$-adic analogue of a simplified version of Jarratt's method which has order of convergence 4. We will also weaken the initial condition $f'(x_1)\not\equiv 0\bmod{p}$, replacing it by the condition \eqref{condi} above. This allows us to begin with an $x_1$ which possibly is a multiple root of the congruence $f(x) \equiv 0 \bmod{p}$. At the end of this article we will briefly discuss Thurston's method which is useful when the initial condition \eqref{condi} is not satisfied. In a nutshell, Thurston's method either tells us after finitely many steps that a Hensel lifting of a given root $x_1$ of the congruence $f(x)\equiv 0 \bmod{p}$ to a solution of $f(x)=0$ in $\mathbb{Z}_p$ does not exist, or it produces a new polynomial $F$ in place of $f$ which meets the above-mentioned initial conditions. In this case we can perform Jarratt's method for $F$, which in turn gives us an approximation to a root of $f$ in $\mathbb{Z}_p$.    
\section{Simplified Jarratt Method for functions on $\rear$}
Jarratt's original method is an iterative numerical method for solving equations $f(x)=0$, where $f$ is a real-valued function defined on intervals in $\rear$. The iteration is defined as 
\begin{equation}\label{JM}
    x_{n+1} = x_n - \frac{1}{2}\frac{f(x_n)}{f'(x_n)} +\frac{f(x_n)}{f'(x_n)-3f'\left( x_n - \frac{2}{3} \frac{f(x_n)}{f'(x_n)}\right)}, \hspace{1cm} \forall n\in\natn.
\end{equation}
Precise initial conditions on $x_1$ for a generalized version of this method on Banach spaces were worked out in \cite{initial-sjm}.
The order of convergence of this method is 4, as established in \cite{jarratt}.

Instead of using the original method, we here use a simplified version. 
The idea is as follows. Using a Taylor approximation (if existent), we find
\begin{equation*}\label{approx}
    f'\left( x_n - \frac{2}{3} \frac{f(x_n)}{f'(x_n)}\right) = f'(x_n) - \frac{2}{3}\cdot\frac{f(x_n)}{f'(x_n)}\cdot f''(x_n) + \frac{2}{9} \left(\frac{f(x_n)}{f'(x_n)}\right)^2f'''(x_n) + \mbox{Error}.
\end{equation*}
Plugging the right-hand side above into \eqref{JM} and neglecting the error, we arrive at the simpler iteration
\begin{equation}\label{SJM}
    x_{n+1} = x_n - \frac{1}{2}\frac{f(x_n)}{f'(x_n)} + \frac{3f(x_n)f'(x_n)^2}{-6f'(x_n)^3 + 6f(x_n)f'(x_n)f''(x_n) - 2f(x_n)^2f'''(x_n)},
\end{equation}
We will refer to this iteration as \textit{Simplified Jarratt Method} (SJM) in the rest of this article. Below we will prove that SJM has order of convergence 4 over the real numbers. If $p>3$, this result carries over to $\mathbb{Q}_p$, as we will see later. 

\begin{theorem}\label{thm:sjm-order}
Assume $I$ is an open interval in $\mathbb{R}$ and $f:I \rightarrow \rear$ is a five times differentiable function. Assume that $\gamma \in I$ satisfies $f(\gamma)=0$ and $f'(\gamma)\not=0$. Then there exists $\delta>0$ such that if $x_1\in (\gamma-\delta,\gamma+\delta)$, the sequence $(x_n)$ defined by \eqref{SJM} converges to $\gamma$ with order $4$. Moreover,
\begin{equation} \label{rateofconvergence}
\lim_{n \rightarrow \infty} \frac{\norm{x_{n+1}-\gamma}}{\norm{x_n-\gamma}^4} = \left| \frac{18f'''(\gamma)f''(\gamma)f'(\gamma)-18f''(\gamma)^3-5f^{(4)}(\gamma)f'(\gamma)^2}{48f'(\gamma)^3}\right|.
\end{equation}
\end{theorem}
\begin{proof}
Let $n\in \mathbb{N}$. Taylor's theorem provides us with the approximations 
\begin{equation} \label{taylors}
\begin{split}
    f(x_n) = & f'(\gamma)\left(e_n+c_2e_n^2+c_3e_n^3+c_4e_n^4+O(e_n^5)\right),\\
    f'(x_n) = & f'(\gamma)\left(1+2c_2e_n+3c_3e_n^2+4c_4e_n^3+O(e_n^4)\right),\\
    f''(x_n)= & f'(\gamma)\left(2c_2+6c_3e_n+12c_4e_n^2+O(e_n^3)\right),\\
    f'''(x_n)= & f'(\gamma)\left(6c_3+24c_4e_n+O(e_n^2)\right),
    \end{split}
\end{equation}
where 
$$
e_n:=x_n-\gamma \quad \mbox{and} \quad c_k:=\frac{1}{k!}\frac{f^{(k)}(\gamma)}{f'(\gamma)}. 
$$
Dividing the first two equations in \eqref{taylors}, we get
\begin{equation}\label{eq:3.6}
    \frac{f(x_n)}{f'(x_n)}=e_n - c_2{e_n^2} + 2(c_2^2 - c_3)e_n^3+(7c_2c_3 - 4c_2^2 - 3c_4)e_n^4+O(e_n^5)
\end{equation}
after a short calculation.
We further write \eqref{SJM} in the form
\begin{equation}\label{eq:3.7}
     x_{n+1} = x_n - \frac{1}{2}\frac{f(x_n)}{f'(x_n)} + \frac{3f(x_n)}{-6f'(x_n)+6\cdot \frac{f(x_n)}{f'(x_n)}\cdot f''(x_n) - 2\left(\frac{f(x_n)}{f'(x_n)}\right)^2f'''(x_n)}.
\end{equation}
Combining \eqref{taylors} and \eqref{eq:3.6}, we approximate the denominator in the last term on the right-hand side of \eqref{eq:3.7} by 
\begin{equation}\label{eq:3.8}
    \begin{split}
        &-6f'(x_n)+6\left(\frac{f(x_n)}{f'(x_n)}\right)f''(x_n) - 2\left(\frac{f(x_n)}{f'(x_n)}\right)^2f'''(x_n)\\
        = & f'(\gamma)\left((- 24c_3^2-48c_2^3+96c_3c_2^2-12c_4c_2)e_n^4+(24c_2^3 - 36c_3c_2)e_n^3+\right.\\ & \qquad \left. (6c_3-12c_2^2)e_n^2-6)+O(e_n^5)\right).
    \end{split}
\end{equation}
From the first line in \eqref{taylors} and \eqref{eq:3.8}, we deduce that
\begin{equation}\label{eq:3.9}
    \begin{split}
        &  \frac{3f(x_n)}{-6f'(x_n)+6\left(\frac{f(x_n)}{f'(x_n)}\right)f''(x_n) - 2\left(\frac{f(x_n)}{f'(x_n)}\right)^2f'''(x_n)}\\
        &= -\frac{e_n}{2}-\frac{c_2}{2}e_n^2+(c_2^2-c_3)e_n^3+\frac{c_4 + 2c_2^3 - 5c_3c_2}{2}\cdot e_n^4 + O(e_n^5),
    \end{split}
\end{equation}
which may be calculated using a suitable computer algebra system.
Now plugging \eqref{eq:3.6} and \eqref{eq:3.9} into \eqref{eq:3.7} and simplifying, we arrive at
\begin{equation*}\label{eq:2.10}
    \begin{split}
     x_{n+1}& =\gamma+e_n-\frac{1}{2}(e_n-c_2e_n^2+2(c_2^2-c_3)e_n^3+(7c_3c_2-4c_2^3-3c_4)e_n^4)+\\
     &\quad \left(-\frac{e_n}{2}-\frac{c_2}{2}e_n^2+(c_2^2-c_3)e_n^3+\frac{c_4 + 2c_2^3 - 5c_3c_2}{2}\cdot e_n^4 + O(e_n^5)\right)\\
     &=\gamma+\frac{9c_3c_2-6c_2^3-5c_4}{2}\cdot e_n^4 + O(e_n^5),
     \end{split}
     \end{equation*}
     and hence
     \begin{equation} \label{en+1}
     e_{n+1}=\rho e_n^4 + O(e_n^5)
\end{equation}
with 
\begin{equation} \label{rhodef}
\rho:=\frac{9c_3c_2-6c_2^3-5c_4}{2}.
\end{equation}
Therefore, if $x_1$ is sufficiently close to $\gamma$, then the sequence $(x_n)$ converges to $\gamma$ with order 4, and we have 
\begin{equation*}
    \lim_{n \rightarrow \infty} \frac{\norm{e_{n+1}}}{\norm{e_n}^4} = |\rho|,
\end{equation*}
which implies \eqref{rateofconvergence}. This completes the proof.
\end{proof}
\section{$p$-adic analogue of SJM}
Now we apply SJM to polynomials in $\mathbb{Z}_p[x]$ subject to certain initial conditions which we work out in detail. We shall prove the following theorem, which is our main result.

\begin{theorem}\label{thm:3.1}
  Let $p>3$ be a prime and $f(x) \in \intz_p[x]$ be a polynomial with leading coefficent a unit in $\mathbb{Z}_p$. 
  Suppose that $x_1 \in \intz_p$ satisfies the condition 
  \begin{equation} \label{ourcondi}
  |f(x_1)|_p<|f'(x_1)|_p^2.
  \end{equation}
  (In particular, $f'(x_1)\not=0$ and $|f(x_1)|_p<1$ and thus $f(x_1)\equiv 0 \bmod{p}$.) Then the sequence $(x_n)$ defined iteratively as
  \begin{equation*}
      x_{n+1} = x_n - \frac{1}{2}\frac{f(x_n)}{f'(x_n)} + \frac{3f(x_n)f'(x_n)^2}{-6f'(x_n)^3 + 6f(x_n)f'(x_n)f''(x_n) - 2f(x_n)^2f'''(x_n)}
  \end{equation*}
  converges to a simple root $\gamma \in \intz_p$ of $f$ with order of convergence 4.
\end{theorem}
As noted in the introduction, in \cite{j.rabago}, a $p$-adic analogue of Olver's method was worked out. This method converges with order 3 and was formulated in \cite{j.rabago} under the stronger initial condition that $|f(x_1)|_p<1$ and $|f'(x_1)|_p=1$, which means that $x_1$ is a simple root of the {\it congruence} $f(x)\equiv 0 \bmod{p}$. In contrast, our condition \eqref{ourcondi} allows us to take $x_1$ as a multiple root of the said congruence.
We point out that \eqref{ourcondi} matches the condition in the following theorem which is a version of the Generalized Hensel Lemma (see \cite{p-adic_book}).
 
\begin{theorem}[Generalised Hensel Lemma]\label{thm:GHL}
Let $f(x)\in \intz_p[x]$ and $a\in \intz_p$ satisfying $|f(a)|_p < |f'(a)|_p^2$. Then
there is a unique $\gamma\in\intz_p$ such that $f(\gamma)=0$ and $|\gamma-a|_p<|f'(a)|_p$. Moreover, $|f'(\gamma)|_p=|f'(a)|_p$. 
\end{theorem}
This shows that a root $a$ of the congruence $f(x)\equiv 0\bmod{p}$ can be lifted uniquely to a simple root $\gamma$ of $f$ in $\mathbb{Z}_p$, provided that $|f(a)|_p<|f'(a)|_p^2$.  Hence, if $x_1=a$ satisfies this condition, then the SJM iteration gives a sequence which converges to precisely this unique lifting $\gamma$ of $a$. 

Theorem \ref{thm:GHL} can be proved using Newton's method (see \cite{k.conrad}). In our proof of Theorem \ref{thm:3.1}, we replace Newton's method by SJM.  
\section{Preliminaries}
For the proof of Theorem \ref{thm:3.1}, we will use the lemma below.
\begin{lemma}\label{lemma:bound}
Let $p>3$ be a prime, $f(x) \in \intz_p[x]$  and $A\in \mathbb{Z}_p$ such that $f'(A)\not=0$. Suppose that 
\begin{equation}\label{assumption}
\left|\frac{f(A)}{f'(A)^2}\right|_p<1.
\end{equation} 
Define
\begin{equation*}
    B:=-\frac{1}{2}\frac{f(A)}{f'(A)}+\frac{3f(A)}{f'(A)-3M},
\end{equation*}
where 
$$
M:= f'(A)-\frac{2}{3}\cdot \frac{f(A)}{f'(A)}\cdot f''(A)+\frac{2}{9}\left(\frac{f(A)}{f'(A)}\right)^2f'''(A).
$$
Then we have 
$$
|B|_p\le \left|\frac{f(A)}{f'(A)}\right|_p.
$$
\end{lemma}
\begin{proof}
We write
\begin{equation}\label{write}
B=\frac{f(A)}{f'(A)}\left(-\frac{1}{2}+\frac{3}{1-3M/f'(A)}\right)
\end{equation}
and 
$$
1-\frac{3M}{f'(A)}=-2+2\cdot \frac{f(A)}{f'(A)^2}\left(f''(A)-\frac{1}{3}\cdot \frac{f(A)f'''(A)}{f'(A)}\right).
$$
From our assumption \eqref{assumption}, it follows that 
$$
\left|1-\frac{3M}{f'(A)}\right|_p=|-2|_p=1
$$
if $p>3$. Hence,
$$
\left|-\frac{1}{2}+\frac{3}{1-3M/f'(A)}\right|_p\le 1
$$
if $p>3$, which together with \eqref{write} implies the claim. 
\end{proof}

\section{Proof of the main result}
Now we turn to the \\ \\
\textbf{Proof of Theorem \ref{thm:3.1}.}
We set
$$
t:=\norm{\frac{f(x_1)}{f'(x_1)^2}}_p
$$
and recall our assumption $|t|_p<1$.
We shall proceed similarly as in the proof of the Generalized Hensel Lemma in \cite{k.conrad} and aim to show the following statements by induction over $n$:
\begin{enumerate}
    \item $\norm{x_n}_p \leq 1$, \label{condition:1}
    \item $\norm{f'(x_n)}_p = \norm{f'(x_1)}_p$, \label{condition:2}
    \item $\norm{f(x_n)}_p\leq \norm{f'(x_1)}_p^2t^{4^{n-1}}$. \label{condition:3}
\end{enumerate}
For $n=1$, these conditions are clearly true.
So let us assume the conditions  to be true for $n$. Now for $n+1$, we have
\begin{equation} \label{appeal}
\norm{x_{n+1}}_p=\norm{x_n+y_n}_p \leq \max\{\norm{x_n}_p,\norm{y_n}_p\},
\end{equation}
where $y_n:=x_{n+1}-x_n$. From Lemma \ref{lemma:bound} we know that 
\begin{equation} \label{important}
\norm{y_n}_p \leq \norm{\frac{f(x_n)}{f'(x_n)}}_p,
\end{equation} where we use statement \eqref{condition:3} together with our assumption $t<1$ above. Further, using \eqref{condition:2}, we deduce that $\norm{y_n}_p \leq \norm{{f(x_n)}/{f'(x_1)}}_p$, and from \eqref{condition:3}, we then get 
\begin{equation} \label{ynineq}
\norm{y_n}_p \leq \norm{\frac{f(x_n)}{f'(x_1)}}_p\leq \norm{f'(x_1)}_pt^{4^{n-1}}\leq 1.
\end{equation}
Finally, from \eqref{condition:1}, we know that $\norm{x_n}_p\leq 1$ and hence we have $\norm{x_{n+1}}_p\leq 1$ by appealing to \eqref{appeal}. Thus, we have established \eqref{condition:1} for $n+1$ in place of $n$.
To prove \eqref{condition:2} for $n+1$, we use the fact that if $F\in \mathbb{Z}_p[x]$, then 
\begin{equation}
 F(x)-F(y)=(x-y)G(x,y)
\end{equation}
for some polynomial $G\in \mathbb{Z}_p[x,y]$ and hence
$$
|F(x)-F(y)|_p\le |x-y|_p
$$
for any $x,y\in \mathbb{Z}_p$. Applying this to $F=f'$, we have
\begin{equation*}
\begin{split}
    \norm{f'(x_{n+1})-f'(x_n)}_p \leq & \norm{x_{n+1}-x_n}_p = \norm{y_n}_p \leq \norm{\frac{f(x_n)}{f'(x_n)}}_p \\ \leq & \norm{f'(x_1)}_pt^{4^{n-1}} < \norm{f'(x_1)}_p,
    \end{split}
\end{equation*}
where we use \eqref{important} and \eqref{condition:3}. 
Now if $\norm{f'(x_{n+1})}_p \not= \norm{f'(x_n)}_p$, then we have
\begin{equation}
  \norm{f'(x_{n+1})-f'(x_n)}_p = \max \{\norm{f'(x_{n+1})}_p,\norm{f'(x_n)}_p\} \geq  \norm{f'(x_n)}_p = \norm{f'(x_1)}_p.  
\end{equation}
Hence we arrive at a contradiction. So $\norm{f'(x_{n+1})}_p = \norm{f'(x_n)}_p = \norm{f'(x_1)}_p$.

To prove \eqref{condition:3} for $n+1$, we use Taylor's formula and recall that $p>3$, obtaining
\begin{equation}
f(x_{n+1})=f(x_n+y_n)=f(x_n)+f'(x_n)y_n+\frac{f''(x_n)}{2!}y_n^2+\frac{f'''(x_n)}{3!}y_n^3+zy_n^4,
\end{equation}
for some $z\in \intz_p$. 
\begin{equation}\label{eq:4.16}
f(x_n)+f'(x_n)y_n+\frac{f''(x_n)}{2!}y_n^2+\frac{f'''(x_n)}{3!}y_n^3 = \frac{f(x_n)^4}{f'(x_n)^3}\frac{A_n}{B_n} = \frac{f(x_n)^4}{f'(x_n)^6}\frac{f'(x_n)^3A_n}{B_n},
\end{equation}
where  
\begin{equation*}
\begin{split} 
A_n:=& a^5d^4-15a^4d^3bc-6a^3d^3b^3+81a^3d^2b^2c^2-189a^2db^3c^3+
18a^2d^2b^4c-\\ & 36ad^2b^6+
 162ab^4c^4+54adb^5c^2-162b^6c^3+108db^7c
\end{split}
\end{equation*} 
and 
$$
B_n:=48(3abc-a^2d-3b^3)^3
$$
with $a:=f(x_n)$, $b:=f'(x_n)$, $c:=f''(x_n)$, $d:=f'''(x_n)$. The above expressions may be calculated conveniently using a suitable computer algebra system.

We note that $|b|_p=|f(x_1)|_p$ by \eqref{condition:2} and $|a|_p< |b|_p^2$ by \eqref{condition:2} and \eqref{condition:3}. It follows that $|A_n|_p\le |b|_p^6$ and 
$|B_n|_p=|b|_p^9$, and hence
\begin{equation*}
\begin{split}
|f(x_{n+1})|_p= &\left|\frac{a^4}{b^3}\cdot \frac{A_n}{B_n}+zy_n^4\right|_p
\leq \max \left\{\left|\frac{a^4}{b^6}\right|_p,|y_n|_p^4\right\}\\
\leq & \left|\frac{a^4}{b^6}\right|_p \cdot \max \left\{1,|b|^2_p\right\}
\leq  \left|\frac{a^4}{b^6}\right|_p \leq |f'(x_1)|_p^2t^{4^n},
\end{split}
\end{equation*}
where for the second inequality above, we use \eqref{important} again, and for the last inequality, we use \eqref{condition:2} and \eqref{condition:3}.
Hence we are done with the induction. 

Using \eqref{ynineq} and $t<1$, it is clear that $(x_n)$ is a Cauchy sequence in $\mathbb{Z}_p$ and hence converges to some $p$-adic integer $\gamma$. 
Further, from  \eqref{condition:2} and \eqref{condition:3}, we get $|f'(\gamma)|_p = |f'(x_1)|_p>0$ and $f(\gamma)=0$. 
Finally, using arguments parallel to those in the proof of \eqref{en+1}, we see that
$$
\left|e_{n+1} - \rho e_n^4\right|_p = O(|e_n|_p^5),
$$ 
where $e_n:=x_n-\gamma$ and $\rho$ as defined in \eqref{rhodef}. Thus we get
\begin{equation}
    \lim_{n \rightarrow \infty} \frac{\norm{e_{n+1}}_p}{\norm{e_n}_p^4} = \norm{\rho}_p
\end{equation}
and hence, $(x_n)$ converges with order 4.
\pushQED{\qed}\qedhere 

\section{Brief concluding remarks}
Two questions remain to be answered.\\ 

(A) Can every zero in $\mathbb{Z}_p$ of $f$ be approximated using SJM, provided we choose $x_1$ suitably?\\

(B) What happens if the initial condition \eqref{ourcondi} on $x_1$ in Theorem \ref{thm:3.1} is not satisfied? Is there a method which produces a suitable $x_1$ satisfying this condition?\\

Regarding question (A), we note that SJM only approximates {\it simple} roots of $f$. If $\gamma$ is a simple root, then clearly there is $x_1$ satisfying the initial condition in Theorem \ref{thm:3.1}, giving rise to a sequence $(x_n)$ converging to $\gamma$: In particular, we may just take $x_1:=\gamma$. If $f$ has multiple roots, then it is easy to produce a polynomial $\tilde{f}$ whose roots are all simple and coincide with the roots of $f$: Take $\tilde{f}=f/\mbox{gcd}(f,f')$, where $\mbox{gcd}(f,f')$ is a greatest common divisor (unique up to multiplication by a unit in $\mathbb{Z}_p$) of $f$ and $f'$ in $\mathbb{Z}_p[x]$. Such a greatest common divisor can be conveniently calculated using the Euclidean algorithm in $\mathbb{Q}_p[x]$. 

To address question (B), we may employ Thurston's method. Here we only describe briefly how this method works. For the details, we refer the interested reader to \cite{thurston}.  Let
\begin{equation}
    \gamma = a_0 + a_1p + a_2p^2 + \cdots
\end{equation}
be a root of $f$ in $\mathbb{Z}_p$. In particular, $f(a_0)\equiv 0 \bmod{p}$. Let 
$$
\gamma_i=a_i+a_{i+1}p+a_{i+2}p^2+ \cdots \mbox{ for } i\in \mathbb{N}_0.
$$ 
Now Thurston's method produces a chain of polynomials $F_0,F_1,F_2,F_3,...$, starting from $F_0=f$, where $F_i$ satisfies the equation $F_i(\gamma_i)=0$. In particular, $F_i(a_i)\equiv 0 \bmod{p}$. Hence, we may proceed as follows to solve $f(x)=0$ in $\mathbb{Z}_p$: First, we find $a_0$ such that $F_0(a_0)=f(a_0)\equiv 0\bmod{p}$. Then we produce $F_1$ and find $a_1$ such that $F_1(a_1)\equiv 0\bmod{p}$, if existent. This $a_1$ may not be unique. Then we produce $F_2$ and find (a not necessarily unique) $a_2$ such that $F_2(a_2)\equiv 0 \bmod{p}$, if existent. We continue this process. If we can continue this process ad infinitum, this gives rise to a sequence $(a_0+a_1p+\cdots +a_np^n)$ which converges to a root $\gamma$ in $\mathbb{Z}_p$. Generally, this sequence converges only with order 1, though.

Thurston proved that if $f$ has only simple roots in $\mathbb{Z}_p$, then for every $a_0$ satisfying $f(a_0)\equiv 0 \bmod{p}$, the above process either stops at some point in which case $a_0$ cannot be lifted to a root $\gamma$ of $f$ in $\mathbb{Z}_p$, or, after finitely many steps, we get an index $n$ for which there exists $a_n$ with $F_n(a_n)\equiv 0 \bmod{p}$ and $F_n'(a_n)\not\equiv 0\bmod{p}$. In this case, we are in the situation of Hensel's Lemma, and $a_n$ therefore lifts uniquely to a root $\gamma_n$ of $F_n$. To apply SJM, we only need the weaker condition $|F_n(a_n)|_p<|F_n'(a_n)|_p^2$ from the Generalized Hensel Lemma. Therefore, if $f$ has only simple roots in $\mathbb{Z}_p$, then we can get all roots of $f$ in $\mathbb{Z}_p$ as follows: We apply Thurston's method initially, starting with an $a_0$ satisfying $f(a_0)\equiv 0 \bmod{p}$ and, as soon as we reach the said index $n$, apply the faster SJM with $x_1=a_n$ and $F_n$ in place of $f$. This gives rise to a sequence $(x_n)$ which converges to a root $\gamma_n\in \mathbb{Z}_p$ of $F_n$. From this, we retrieve a root
$$
\gamma=a_0+a_1p+...+a_{n-1}p^{n-1}+\gamma_np^n \in \mathbb{Z}_p
$$ 
of $f(x)$. 

One last brief remark on practical applications of SJM. This method produces a sequence whose elements lie in $\mathbb{Z}_p$. For practical applications, however, only sequences in $\mathbb{Z}$ are useful since we cannot encode an infinite $p$-adic representation using a computer. Therefore, to implement SJM in practice, one needs to cut off the $p$-adic representation of $x_n$ at an appropriate point, getting an integer $x_n'$. Then one calculates $x_{n+1}$ (and hence $x_{n+1}'$) using $x_n'$ in place of $x_n$. This will give a sequence $(x_n')$ of integers converging to our root $\gamma\in \mathbb{Z}_p$ of $f$.     

\nocite{sixth-order-jarratt}
\nocite{p-adic-numbers}
\nocite{ultrametric}
\nocite{reddy}
\nocite{koblitz}
\nocite{Robert}
\nocite{axel}
\bibliography{ref}
\bibliographystyle{plain}
\end{document}